\documentclass[11pt]{amsart}
\usepackage{amsmath,amssymb,enumerate,etoolbox,mathrsfs,tikz,amscd}
\usepackage[all]{xy}
\newtheorem{thm}{Theorem}[section]

\newtheorem{prop}[thm]{Proposition}

\newtheorem{lem}[thm]{Lemma}
\theoremstyle{definition}

\newtheorem*{remark*}{Remark}

\numberwithin{equation}{section}

\newcommand{\C}{\mathbb{C}}

\renewcommand{\P}{\mathbb{P}}

\newcommand{\Z}{\mathbb{Z}}

\newcommand{\Gal}{{\rm Gal}}

\title{Genus zero tensor products}
\author[M.~J.~Larsen]{Michael J.~Larsen}
\address{Department of Mathematics, Indiana University, Bloomington, IN 47405, USA}
\email{mjlarsen@iu.edu}
\author[Y.~Shi]{Yue Shi}
\address{Department of Mathematics, Indiana University, Bloomington, IN 47405, USA}
\email{shi10@iu.edu}
\begin{document}
\begin{abstract}
Let $L$ and $M$ be finite extensions of $K = \C(t)$.  If $L\otimes _K M$ is a field of genus $0$, then at least one of $L$ and $M$
is ramified over at most four valuations of $K$.
\end{abstract}
\thanks{ML was partially supported by NSF grant DMS-2401098 and by the Simons Foundation.}
\maketitle

\section{Introduction}

Finite extensions of $K=\C(t)$ which are of genus zero, i.e. isomorphic over $\C$ to $K$ itself, are in various senses special. In particular, there are strong
constraints on the possible Galois groups $\Gal(L^s/K)$, where $L^s$ denotes the splitting field of $L/K$ (see \cite{A,GT,LS}).

Pairs $(L,M)$ of finite $K$-extensions, where
$L\otimes_K M$ is of genus zero, satisfy even stronger group-theoretic constraints. Unlike in the case of a single extension, these constraints
are strong enough to limit ramification.
In \cite{ILZ}, it is proved that at least one of $L$ and $M$
is ramified over $\le 5$ valuations of $K$. It is further predicted that the bound $5$ can be improved.
In this paper, we prove that $4$ is the optimal bound.

%This does limit the ramification, at least for the extension of lower degree.
%By \cite[Theorem 2.1]{ILZ}, if $[M:K]\ge [L:K]$, $L$ can be ramified over no more than $5$ valuations
%of $K$. We prove here that this result is optimal. When $[M:K]$ is sufficiently large, the maximum number of ramified places of $L/K$ can be improved to $4$. This bound cannot be further improved.

Via the Riemann-Hurwitz theorem and the well known correspondence between
transitive permutation representations of fundamental groups of punctured Riemann surfaces and branched covers of those surfaces \cite[III Corollary 4.10]{M},
the question considered in this paper reduces to a set of problems about permutation groups. We use elementary inequalities to determine the possibilities for local monodromy which are numerically consistent with genus zero.
Usually, all that matters is the number of fixed points and the number of orbits of the various local monodromy elements.
Once $[L:K]$ and $[M:K]$ are sufficiently large (we assume $32$ in this paper), this can be done in a fairly straightforward way, but the smaller values are more troublesome and are treated by computer-assisted case analysis.
In particular, we deal with the case $[L:K]=[M:K]=6$ by enumerating by computer all pairs of $5$-tuples of conjugacy classes
in $\mathsf{S}_6$, for which $L$, $M$ and $L\otimes_K M$ all have genus $0$.
At the end of the classification process, the
A-D-E classification makes a surprising appearance.

Next we consider whether there are actual field extensions $L$ and $M$ of $K$ satisfying the specified local monodromy conditions; this amounts to searching for $5$-tuples of permutations which lie in specified conjugacy classes, 
multiply to $1$, and generate transitive groups.  Finally, assuming that $L$ and $M$ can be constructed individually, the question remains as to whether $L\otimes_K M$, which formally has genus $0$, is actually a field.
It turns out that in every case, the answer is no. For the cases of type $D_n$ this can be shown directly ($L$ and $M$ must each contain a quadratic subfield  ramified only over two specified valuations of $K$, so the tensor product of those subfields has zero-divisors.)
Lacking a conceptual proof that covers the $E_n$ cases, we 
make use of machine-assisted case analysis.  For $E_8$, in particular, these computations require some care, since the number of cases to be considered is enormous.
\begin{thm}
If $L$ and $M$ are finite extensions of $K$ such that $L\otimes_K M$ is of genus zero,
then at least one of $L$ and $M$ is ramified over no more than $4$ valuations of $K$.

\end{thm}

In the last section of this paper, we use pairs of Latt\`es maps to show that this theorem is sharp.
\section{Numerical conditions}

Let $\ell=[L:K]$ and $m=[M:K]$.  We assume without loss of generality that $m\ge \ell$.
Let $x_1,\ldots,x_n$ denote the valuations of $K$ over which $L$ or $M$ is ramified, arranged so that the points ramified in $L$ come first. We already know by \cite{ILZ} that $L/K$ is unramified over $x_i$ for $i\ge 6$.
Let $e_{i,1},e_{i,2},\ldots$ denote the ramification indices of the valuations $y_{i,1},y_{i,2},\ldots$ of $L$ lying over $x_i$
and  $f_{i,1},f_{i,2},\ldots$ denote the ramification indices of the valuations $z_{i,1},z_{i,2},\ldots$ of $M$ lying over $x_i$.  
We write $\pi_L$ for the sequence of partitions of $\ell$ whose $i$th term has parts $e_{i,1},e_{i,2},\ldots$,
and likewise for $\pi_M$ and $f_{i,1},f_{i,2},\ldots$.

As 
$$\C((t^{1/e}))\otimes_{\C((t))} \C((t^{1/f}))\cong \C(t^{1/[e,f]})^{(e,f)},$$
the degree of the ramification divisor of $LM/L$ restricted to $y_{i,j}$ is
$$\sum_k (e_{i,j},f_{i,k})\Bigl(\frac{[e_{i,j},f_{i,k}]}{e_{i,j}}-1\Bigr).$$
Therefore, the ramification divisor of $LM/L$ has degree $\sum_i E_i$, where
\begin{equation}
\label{LM/L}
E_i = \sum_{j,k}(e_{i,j},f_{i,k})\Bigl(\frac{[e_{i,j},f_{i,k}]}{e_{i,j}}-1\Bigr).
\end{equation}
Likewise, the ramification divisor of $LM/M$ has degree $\sum_i F_i$, where
$$F_i = \sum_{j,k}(e_{i,j},f_{i,k})\Bigl(\frac{[e_{i,j},f_{i,k}]}{f_{i,k}}-1\Bigr).$$

Let $R_i = \sum_j (e_{i,j}-1)$ (resp. $S_i = \sum_k (f_{i,k}-1)$) denote the degree of the ramification divisor
of $L/K$ (resp. $M/K$) lying over $x_i$. By L\"uroth's theorem, if $L\otimes_K M$ is a field of genus zero, then $L$ and $M$ are also of genus zero,
so by the Riemann-Hurwitz formula,
\begin{equation}
\label{RHL}
\sum_i R_i = 2\ell-2
\end{equation}
and
\begin{equation}
\label{RHM}
\sum_i S_i = 2m-2.
\end{equation}

Let $a_i$ (resp. $b_i$) denote the number of terms $j$ (resp. $k$) such that $e_{i,j}=1$ (resp. $f_{i,k}=1$).
By \eqref{LM/L}, the contribution to $E_i$ from terms $(j,k)$ where $e_{i,j} = 1$ is 
$a_i\sum_k (f_{i,k}-1) = a_i S_i$. Likewise, the contribution to $F_i$ from terms $(j,k)$ where $f_{i,k}=1$
is $b_i R_i$.  Again by Riemann-Hurwitz, we have
\begin{equation}
\label{e=1}
\sum_i a_i S_i \le \sum_i E_i = 2[LM:L]-2 = 2m-2,
\end{equation}
and
\begin{equation}
\label{f=1}
\sum_i b_i R_i \le \sum_i F_i = 2[LM:M]-2 = 2\ell-2.
\end{equation}
For each $i$, we have
\begin{equation}
\label{R-ineq}
\begin{split}
a_i+R_i &= a_i+\sum_{\{j: e_{i,j}\ge 2\}} (e_{i,j}-1) \le a_i+\sum_{\{j: e_{i,j}\ge 2\}} e_{i,j} =\sum_j e_{i,j}=\ell\\
& \le a_i+2\sum_{\{j: e_{i,j}\ge 2\}} (e_{i,j}-1) \le a_i+2R_i.,
\end{split}
\end{equation}
where the first inequality is strict unless $R_i=0$.
Likewise
\begin{equation}
\label{S-ineq}
b_i+S_i\le m\le b_i+2S_i,
\end{equation}
where the first inequality is strict unless $S_i=0$.
%a_i+2R_i =  a_i+2\sum_{\{j: e_{i,j}\ge 2\}} (e_{i,j}-1) \ge  a_i+\sum_{\{j: e_{i,j}\ge 2\}} e_{i,j}=\sum_j e_{i,j}=\ell,$$
We say $L/K$ (resp. $M/K$) is \emph{homogeneous of index $r$ over $x_i$} if $a_{i,j}=r$ for all $j$ (resp. $b_{i,k}=r$ for all $k$).

\begin{lem}
\label{inhomo}
Assume  $L$ and $M$ are not both homogeneous of the same index over $x_i$.
If $\ell,m \ge 4$, either $E_i\ge 2$ or $F_i\ge 2$, and if $\ell,m \ge 7$,
either $E_i\ge 3$ or $F_i\ge 3$.
\end{lem}

\begin{proof}
We have $[e,f]-e\ge 3$ unless $f$ divides $e$ or $(e,f)\in\{(1,2),(1,3),(2,4)\}$, and $[e,f]-f\ge 3$ unless $e$ divides $f$ or
$(e,f)\in\{(2,1),(3,1),(4,2)\}$. Therefore, $\max(E_i,F_i)\le 1$ implies that there is at most one ordered pair $(j,k)$ such that 
$e_{i,j}\neq f_{i,k}$, and for this pair $\{e_{i,j},f_{i,k}\} = \{1,2\}$.
Likewise, 
$\max(E_i,F_i)\le 2$ implies that there are at most two ordered pairs $(j,k)$ such that $e_{i,j}\neq f_{i,k}$, and for these pairs $\{e_{i,j},f_{i,k}\} = \{1,2\}$
or there is at most one ordered pair $e_{i,j}\neq f_{i,k}$, and for this pair $\{e_{i,j},f_{i,k}\}$ is $\{1,3\}$ or $\{2,4\}$.

Assume $E_i,F_i\le 1$ and (without loss of generality) $L$ is not homogeneous.  Then $e_{i,j}\in\{1,2\}$ for all $j$, and both values must occur at least once.  If there are two values $j$ with $e_{i,j}=1$, then $f_{i,k}=1$ for all $k$, so there can be at most one $k$-value, and $m=1$.  If there are two values $j$ with $e_{i,j}=2$, then $f_{i,k}=2$ for all $k$, so there can be at most one $k$-value, and $m=2$.  

Assume $E_i,F_i\le 2$ and $L$ is not homogeneous.  Then $\{e_{i,1},e_{i,2},\ldots\}$ must be $\{1,2\}$, $\{1,3\}$, or $\{2,4\}$. In the case $\{1,2\}$, since $\ell\ge 7$, at least one element of $\{1,2\}$ must occur as $e_{i,j}$ for at least $3$ values of $j$, and $f_{i,k}$ must then take this value for all $k$.
There can be at most $2$ $k$-values, therefore, and $m\ge 4$.  
In the case $\{1,3\}$, since $\ell\ge 5$, at least one element of $\{1,3\}$ must occur for at least $2$ values of $j$,
and $f_{i,k}$ must take this value for all $k$, implying $m\ge 3$.
In the case $\{2,4\}$, since $\ell\ge 7$, at least one element of $\{2,4\}$ must occur for at least $2$ values of $j$,
and $f_{i,k}$ must take this value for all $k$, implying $m\ge 4$.

\end{proof}

%Let $m\ge \ell$ be positive integers with $m>120$. Suppose that for $i=1,2,3,4,5$ we have non-negative integers  
%$a_i>0$, $b_i$, $R_i$, and $S_i$.  Consider the set of conditions
%%
%\begin{equation}
%\label{conditions}
%\begin{split}
%a_i + R_i \le \ell &\le a_i+2R_i,\ 1\le i\le 5\\
%b_i + S_i \le m &\le b_i+2S_i,\ 1\le i\le 5\\
%\sum_i R_i &= 2\ell-2\\
%\sum_i S_i &\le 2 m-2\\
%\sum_i b_i R_i &\le 2\ell-2\\
%\sum_i a_i S_i &\le 2m-2\\
%\end{split}
%\end{equation}
%%
\begin{lem}
\label{basic}
Suppose $m\ge \ell > 0$, $a_i>0$, $b_i$, $R_i$, and $S_i$ are non-negative integers
satisfying the conditions \eqref{RHL}, \eqref{RHM}, \eqref{e=1}, \eqref{f=1}, \eqref{R-ineq}, and
\eqref{S-ineq}. Then
after suitable renumbering of the indices, we have $R_1=1$, $S_2<m/3$, and $m\le 2\ell-6$.
%
%$$R_1=R_2=1,\ 1\le S_1+S_2\le 2,\ m\le \ell+1.$$
%
\end{lem}

\begin{proof}
By \eqref{R-ineq}, \eqref{RHL},  \eqref{S-ineq}, and \eqref{RHM}, we have
\begin{equation}
\label{plus 4}
\begin{split}
\sum_{i=1}^5 a_i &\ge \sum_{i=1}^5 (\ell-2R_i) = 5\ell - 2\sum_{i=1}^5 R_i = \ell+4;\\
\sum_{i=1}^5 b_i &\ge \sum_{i=1}^5 (m-2S_i) = 5m - 2\sum_{i=1}^5 S_i = m+4\ge \ell+4.
\end{split}
\end{equation}
By \eqref{f=1}, this implies
$$2\ell-2 \ge \sum_{i=1}^5 b_i R_i \ge(\ell+4) \min_i R_i,$$
and therefore $\min_i R_i=1$. Reordering the indices, we are justified in assuming $R_1=1$, so 
$a_1=\ell-2$, and
\begin{equation}
\label{six}
a_2+a_3+a_4+a_5\ge 6.
\end{equation}
By \eqref{e=1},%
$$\sum_{i=2}^5 a_i S_i \le 2m-2,$$
so $S_i < m/3$ for some $i\ge 2$; without loss of generality, we can take $i=2$.
By \eqref{plus 4} and \eqref{f=1},
$$m+4\le \sum_{i=1}^5 b_i \le \sum_{i=1}^5 b_i R_i \le 2\ell-2,$$
so $m\le 2\ell-6$.  
\end{proof}

\begin{lem}
\label{m>6}
Let $m\ge 7$.  Under the hypotheses of Lemma~\ref{basic}, after renumbering the indices, either
$R_1=R_2=1$, $1\le S_1+S_2\le 2$, and $m\le \ell+1$, or 
$(\ell,m,R_1,R_2,S_1,S_2,a_1,a_2,b_1,b_2)$ is one of the following:
\begin{enumerate}
\item $(7,7,1,2,1,2,5,3,5,3)$
\item $(7,7,1,3,2,2,5,1,3,3)$
\item $(8,8,1,2,1,2,6,4,6,4)$
\item $(8,10,1,1,2,1,6,6,8,6)$
\item $(8,10,1,1,2,1,6,6,6,8)$
\item $(8,10,1,1,0,3,6,6,10,4)$
\item $(9,9,1,3,1,3,7,3,7,3)$
\item $(9,9,1,2,0,3,7,5,9,3)$
\item $(10,10,1,2,0,3,8,6,10,4)$
\end{enumerate}
\end{lem}

\begin{proof}

Since $(\ell-2)S_1 = a_1 S_1 \le 2m-2 $, which is $\le 4\ell-14$
by Lemma~\ref{basic}, we have $S_1 \le 3$, so by \eqref{S-ineq}, $b_1 \ge m-6$.

By computer, we  examine all tuples $(\ell,m,R_1,R_2,S_1,S_2,a_1,a_2,b_1,b_2)$ satisfying
$R_1=1$, $S_2<m/3$, $(m+6)/2\le \ell\le m$, $7\le m \le 31$,
\eqref{R-ineq} and \eqref{S-ineq} for $i=1,2$, $a_1S_1+a_2S_2 \le 2m-2$,
and $b_1R_1+b_2R_2\le 2\ell-2$.
We confirm that in every case other than those listed above, $R_2=1$, $m\le \ell+1$, and $1\le S_1+S_2\le 2$.

We may therefore assume $m\ge 32$, which means $\ell\ge \frac{m+6}2 \ge 19$.
As $S_2 < m/3$, $b_2 > m/3$ by \eqref{S-ineq}, so $R_2 \le 5$ by \eqref{f=1}, and $a_2 \ge \ell-10$ by \eqref{R-ineq}.
As
$$a_1 S_1 + a_2 S_2 \le 2m-2  \le 4\ell-14$$
and $\ell\ge 19$, we have $S_2\le 6$, and $b_2 \ge m-12$.  Thus, 
$$2\ell-2\ge b_1 R_1 + b_2 R_2 \ge b_1+b_2 \ge (m-6)+(m-12) = 2m-18,$$
so $m \le \ell+8$, and $\ell \ge 24$.  As $b_1R_1+b_2R_2\le 2m-2$, we have $R_2=1$. 
Since
$$2\ell+14 \ge 2m-2 \ge a_1 S_1 + a_2 S_2 = (\ell-2)(S_1+S_2),$$
we have $S_1+S_2 \le 2$.
The case $S_1=S_2=0$ is ruled out
since $b_1=b_2=m$ would imply 
$2\ell-2 \ge b_1R_1 + b_2R_2 = 2m$, which is impossible.
As
$$2\ell-2 \ge b_1R_1+b_2R_2\ge b_1+b_2 \ge (m-2S_1)+(m-2S_2) \ge 2m-4,$$
we have $m\le \ell+1.$
\end{proof}

\begin{lem}
\label{m=6}
Suppose $m=6$, and the hypotheses of Lemma~\ref{basic} are satisfied. Then after reindexing, we must have
$\pi_L=(1^4 2,1^4 2, 1^22^2,2^3,2^3)$ and $\pi_M = (1^6,1^4 2,2^3,2^3,2^3)$.
\end{lem}

\begin{proof}
If $m=6$, the inequality $\ell\le m\le 2\ell-6$ implies $\ell=6$.
We choose an order on the partitions of $6$ and enumerate by machine all ordered pairs of ordered $5$-tuples of partitions of $6$ such that the partitions in the first $5$-tuple have at most $5$ parts and are weakly increasing
with respect to our order, such that 
$R_i\ge 1$ for $1\le i\le 5$ and satisfying the following conditions:
$$\sum_{i=1}^5\sum_j (e_{i,j}-1) \le10,$$
$$ \sum_{i=1}^5\sum_k(f_{i,k}-1) \le 10,$$
$$\sum_{i=1}^5 E_i \le 10,$$
$$\sum_{i=1}^5 F_i \le 10.$$
The only solutions are the one given above and the equivalent solution 
$$(\pi_L,\pi_M) = ((1^4 2,1^4 2, 1^22^2,2^3,2^3),(1^4 2,1^6,2^3,2^3,2^3)).$$
\end{proof}

\begin{thm}
\label{local}
If $L/K$ and $M/K$ are non-trivial finite extensions isomorphic to $\C((t))$ as $\C$-algebras satisfying $\sum_i E_i = 2m-2$
and $\sum_i F_i = 2\ell-2$ with $m\ge \ell$ and $L/K$ ramified at $\ge 5$ points, then, after suitably reindexing, $(\pi_L,\pi_M)$ must be one of the following:
\begin{enumerate}
%\item $((1^6 2,1^4 2^2,2^4,2^4,2^4),(1^8 2,1^6 2^2, 2^5,2^5,2^5))$
%\item $((1^{4n-2},1^{4n-2},2^{2n},2^{2n},n^4),((1^{4n-2},1^{4n-2},2^{2n},2^{2n},n^4))$ for some $n\ge 2$
\item $((1^{4n-2}2,1^{4n-2}2,2^{2n},2^{2n},n^4,1^{4n}),(1^{4n-2}2,1^{4n},2^{2n},2^{2n},n^4,1^{4n-2}2))$ for some $n\ge 2$
%\item $((1^{22}2,1^{22}2,2^{12},3^{8},3^{8}),(1^{22}2,1^{22}2,2^{12},3^{8},3^{8}))$
\item $((1^{22}2,1^{22}2,2^{12},3^{8},3^{8},1^{24}),(1^{22}2,1^{24},2^{12},3^{8},3^{8},1^{22}2))$
%\item $((1^{46}2,1^{46}2,2^{24},3^{16},4^{12}),(1^{46}2,1^{46}2,2^{24},3^{16},4^{12}))$
\item $((1^{46}2,1^{46}2,2^{24},3^{16},4^{12},1^{48}),(1^{46}2,1^{48},2^{24},3^{16},4^{12},1^{46}2))$
%\item $((1^{118}2,1^{118}2,2^{60},3^{40},5^{24}),(1^{118}2,1^{118}2,2^{60},3^{40},5^{24}))$
\item $((1^{118}2,1^{118}2,2^{60},3^{40},5^{24},1^{120}),(1^{118}2,1^{120},2^{60},3^{40},5^{24},1^{118}2))$
\end{enumerate}
\end{thm}

\begin{proof}
As $\ell \le m\le 2\ell-6$, we have $\ell\ge 6$.   If $\ell=6$, then $m=6$, and by Lemma~\ref{m=6}, $M$ is ramified over $K$
only at $x_2$, $x_3$, $x_4$, and $x_5$.  We therefore assume $\ell\ge 7$.

Suppose first that $R_1=R_2=1$ and $S_1+S_2=2$.  
Then $E_1+E_2=2m-4$, $F_1+F_2=2\ell-4$, and $\ell\le m\le \ell+1$.  
There cannot be a sixth point $x_6$ over which $M/K$ ramifies, since that would imply $F_6\ge \ell$,
so $F_1+F_2+F_6 > 2\ell-2$.
By Lemma~\ref{inhomo}, therefore, for $3\le i\le 5$,
$L/K$ and $M/K$ must both be homogeneous of equal index $r_i\ge 2$ at $x_i$.
We assume without loss of generality that $r_5\ge r_4\ge r_3$.
Since $r_3$, $r_4$, and $r_5$ divide $\ell$ and $m$ and $|\ell-m|\le 1$, we have $\ell=m$.
Equation \eqref{RHL} implies
$$2\ell-2 = 1+1+\ell-\frac\ell{r_3} + \ell-\frac\ell{r_4} + \ell-\frac\ell{r_5},$$
so
$$\frac 1{r_3} + \frac 1{r_4} + \frac 1{r_5} > 1.$$
This is exactly the inequality that arises in classifying simply laced Dynkin diagrams with a trivalent vertex \cite[\S 11.4 (10)]{Hump},
and the only possibilities for $(r_3,r_4,r_5)$ are $(2,2,n)$ for some $n\ge 2$, $(2,3,3)$, $(2,3,4)$, and $(2,3,5)$, corresponding to $D_{n+2}$, $E_6$, $E_7$, and $E_8$ respectively.
These give, respectively, the solutions (2), (3), (4), and (5).

This leaves the cases (1)--(9) of Lemma~\ref{m>6}. In every case,
we have $a_1 S_1 + a_2 S_2 \ge 2m-3$
and $b_1 R_1 + b_2 R_2 \ge 2\ell-3$.  By Lemma~\ref{inhomo}, for $3\le i\le 5$, $L$ and $M$ must be homogeneous over $x_i$ of equal index.  In cases (1) and (2), this index must be $7$, contrary to \eqref{RHL}.
In case (3), $R_3,R_4,R_5\ge 4$, which is again inconsistent with \eqref{RHL}.
In cases (4) and (5) of Lemma~\ref{m>6} a case analysis by computer shows that a transitive subgroup of $\mathsf{S}_{10}$ cannot be generated by $5$ elements, of types $1^8 x$, $1^6 2^2$, $2^5$, $2^5$, and $2^5$ which multiply to $1$.
In case (6), 
$$(\pi_L,\pi_M)=((1^6 2, 1^6 2, 2^4, 2^4, 2^4), (1^{10}, 1^4 2^3, 2^5, 2^5, 2^5)),$$
which means that $M/K$ is ramified only over $x_2,x_3,x_4,x_5$.
In cases (7) and (8), $R_3,R_4,R_5\ge 6$, is inconsistent with \eqref{RHL},
In case (9), 
$$(\pi_L,\pi_M)=((1^8 2, 1^6 2^2, 2^5, 2^5, 2^5), (1^{10},1^6 2^2, 2^5, 2^5, 2^5)),$$
and, again, $M/K$ is ramified only over $x_2,x_3,x_4,x_5$.
\end{proof}

\section{Permutation groups}

\begin{prop}
\label{dn}
Let $X$ be a set with $4n$ elements, and let $g_1,g_2,g_3,g_4$ denote permutations of $X$ such that
$g_1$ and $g_2$ each have $2n$ orbits of size $2$, $g_3$ has $4$ orbits of size $n$, and $g_4$ and
$g_1 g_2 g_3 g_4$ are transpositions. If $G=\langle g_1,g_2,g_3,g_4\rangle$ acts transitively on $X$, then
$X$ admits a partition $X= Y\sqcup Z$ such that 
\begin{equation}
\label{imprimitive}
g_1(Y) = g_2(Y) = g_3(Z) = g_4(Z) = Z
\end{equation}
\end{prop}

\begin{proof}
For each $x\in X$, we consider the orbit $x^H$ of $x$ under the dihedral group $H=\langle g_1,g_2\rangle$. Since $g_1$ and $g_2$ act on $x^H$ without fixed points, $|x^H|$ must be even, and $x^H$ must decompose 
into two $\langle g_1g_2\rangle$
orbits of equal length. In other words, the $\langle g_1 g_2\rangle$-orbits of $X$ 
can be paired in such a way that $g_1$ and $g_2$ both send each such orbit to its partner.

Writing $g_5 = g_1 g_2 g_3 g_4$, we have $g_1 g_2 =  g_5 g_4 g_3^{-1}$.
If $g = (x_1\, x_2)$ is a transposition and $h$ any permutation, the orbits of $\langle h\rangle$ not containing $x_1$ or $x_2$ are $\langle gh\rangle$-orbits.
If $x_1$ and $x_2$ belong to different $\langle h\rangle$-orbits, the union of these orbits form a single $\langle gh\rangle$-orbit, while otherwise, the $\langle h\rangle$-orbit containing $x_1$ and $x_2$ splits into two $\langle gh\rangle$-orbits, of lengths $a$ and $b$ if $a$ and $b$ are the smallest positive integers such that $h^a(x_1) = x_2$ and $h^b(x_2) = x_1$.

Since $g_3^{-1}$ has type $n^4$, the type of $g_4 g_3^{-1}$ is either $n^2(2n)$ or $a(n-a) n^3$ for some $a\in [1,n-1]$.
In the former case, the type of $g_5 g_4 g_3^{-1}$ could be $(2n)^2$, $(3n)n$, $b(n-b)n(2n)$, or $n^2 c (2n-c)$,
where $b\in [1,n-1]$, $c\in [1,2n-1]$. In the latter case, it could be of type $(n-a)n^2(n+a)$, $an^2(2n-a)$, $a(n-a)n(2n)$, $n^4$,
$b(a-b)(n-a)n^3$, $ac(n-a-c)n^3$, or $ad(n-d)(n-a)n^2$, where $b\in [1,a-1]$, $c\in [1,n-a-1]$, and $d\in [1,n-1]$.
The types in which orbit lengths pair up are $(2n)^2$, $n^4$, and $a^2 (n-a)^2 n^2$.  We consider each of these possibilities in turn.
\vskip 10pt\noindent 
Case $(2n)^2$.  Let $X_1$, $X_2$, $X_3$, and $X_4$ denote the four orbits of $g_3$, numbered so that $g_4 = (x_1\, x_2)$
and $g_5 = (x_3\, x_4)$, with $x_i\in X_i$. Let $Y=X_1\sqcup X_2$ and $Z = X_3\sqcup X_4$.
$g_3$, $g_4$, and $g_5$ preserve $Y$ and $Z$, the same is true for $g_1 g_2$, so $g_1$ and $g_2$ each map $Y$ to $Z$ and $Z$ to $Y$.
\vskip 10pt\noindent 
Case $n^4$.  First we consider the subcase that $g_4 g_3^{-1}$ is of type $n^2(2n)$. We write $g_4 = (x_1\,x_2)$
and number the orbits $X_i$ of $g_3$ so that $x_1\in X_1$ and $x_2\in X_2$.  Then $g_5$ must transpose two elements of $X_1\cup X_2$.  The orbit of $\langle g_1,g_2\rangle$ containing $X_1$ must contain one other $X_i$. If $i=2$, then $X_1\sqcup X_2$ is stable under the action of $G$, contrary to the assumption of transitivity. Otherwise, $Y=X_1\sqcup X_2$
and $Z = X_3\sqcup X_4$ satisfy \eqref{imprimitive}.

Second, we consider the subcase that $g_4 g_3^{-1}$ is of type $a(n-a)n^3$. In that case, we may assume $g_4 = (x_1\,y_1)$,
with $x_1,y_1\in X_1$. Choosing $i>1$ such that $Y=X_1\sqcup X_i$ is $\langle g_1,g_2\rangle$-stable, we see that $Y$ is $G$-stable, contrary to transitivity.

\vskip 10pt\noindent 
Case $a^2 (n-a)^2 n^2$. Renumbering, we may assume $g_4 = (x_1\,y_1)$, $g_5=(x_2\,y_2)$ with $x_1,y_1\in X_1$
and $x_2,y_2\in X_2$. Then $Y=X_3\sqcup X_4$ forms a single $\langle g_1,g_2\rangle$-orbit, so $Y$ is $G$-stable,
contrary to transitivity.
\end{proof}

\begin{prop}
\label{e6}
Let $g_1$, $g_2$, $g_3$, and $g_4$ be elements of $\mathsf{S}_{24}$ of type $2^{12}$, $3^8$, $3^8$, and $1^{22}2$
respectively such that $g_1 g_2 g_3 g_4$ is of type $1^{22}2$. If $\langle g_1,g_2,g_3,g_4\rangle$ is transitive, then
there exists a partition of $\{1,2,\ldots,24\}$ into $2$-element sets such that all $g_i$ respect the partition and $g_1 g_2 g_3$ and $g_4$ send each $2$-element set to itself.
\end{prop}

\begin{proof}
As $g_1 g_2$ can be expressed as a product of three terms, two of which are transpositions and one of  which is of type $3^8$, it follows that $g_1 g_2$ must be of type $3^46^2 $,  $3^59$, , $123^5 6$,
$13^6 5$, $23^6 4$, $123^7$, or $3^8$.  We make a computer search for pairs of $(g_1,g_2)$ of type $(2^{12},3^8)$ such that $g_1 g_2$ is of one of these types.
In doing this, we may assume without loss of generality that 
$$g_2 = (1,2,3)(4,5,6)(7,8,9)\cdots(22,23,24).$$
The action of the centralizer $C_{\mathsf{S}_{24}}(g_2)$ on elements of type $2^{12}$ partitions the conjugacy class $C_{2^{12}}$ of such elements into orbits, and we choose 
at least one representative for $g_1$ from each such orbit.
To do this, we consider 
$$g_1 = (a_1,b_1)(a_2,b_2)\cdots(a_{12},b_{12}).$$
Let $X_k:=\{a_1,a_2,\ldots,a_k,b_1,b_2,\ldots,b_k\}$, and let $\xi_k$ denote the maximum of $\bigl\lfloor \frac {x+2}3\bigr\rfloor$ for $x\in X_k$.
Consider the conditions that $a_{k+1}$ is the smallest integer in $[1,24]$ which does not belong to $X_k$, and if $b_{k+1} > 3\xi_1$, then $b_{k+1} = 3\xi_k+1$. 
It is easy to see that every $C_{\mathsf{S}_{24}}(g_2)$-orbit of $C_{2^{12}}$ is represented by such an element.
There are six solutions for which transpositions $g_1$ and $g_1 g_2 g_3 g_4$ can be chosen with $g_3 \in C_{3^8}$:
\begin{gather*}
(1,4)(2,7)(3,10)(5,12)(6,8)(9,13)(11,16)(14,19)(15,22)(17,24)(18,20)(21,23),\\
(1,4)(2,7)(3,10)(5,12)(6,13)(8,16)(9,11)(14,19)(15,22)(17,24)(18,20)(21,23),\\
(1,4)(2,7)(3,10)(5,12)(6,13)(8,16)(9,19)(11,14)(15,22)(17,24)(18,20)(21,23),\\
(1,4)(2,7)(3,10)(5,13)(6,8)(9,11)(12,16)(14,19)(15,22)(17,24)(18,20)(21,23),\\
(1,4)(2,7)(3,10)(5,13)(6,8)(9,15)(11,16)(12,19)(14,22)(17,24)(18,20)(21,23),\\
(1,4)(2,7)(3,10)(5,13)(6,16)(8,19)(9,11)(12,21)(14,22)(15,17)(18,24)(20,23).\\
\end{gather*}

 In every case, $g_1 g_2$ is of type $3^86^2$. Using GAP, we check that in every case, the group is imprimitive, preserving a decomposition of 
$\{1,2,\ldots,24\}$
into a set of $12$ pairs, on which it acts transitively.
\end{proof}

\begin{prop}
\label{e7}
Let $g_1$, $g_2$, $g_3$, and $g_4$ be elements of $\mathsf{S}_{48}$ of type $2^{24}$, $3^{16}$, $4^{12}$, and 
$1^{46}2$
respectively such that $g_1 g_2 g_3 g_4$ is of type $1^{46}2$. If $\langle g_1,g_2,g_3,g_4\rangle$ is transitive, then
there exists a partition of $\{1,2,\ldots,48\}$ into $2$-element sets such that all $g_i$ respect the partition and $g_1 g_2 g_3$ and $g_4$ send each $2$-element set to itself.
\end{prop}

\begin{proof}
We use the same method as before, with $g_2=(1,2,3)\cdots (46,47,48)$.  As $g_1 g_2$ must be the product of an element of $C_{4^{12}}$ and two transpositions, it must be of type
$4^88^2$, $4^9(12)$, $134^98$, $2^2 4^9 8$, $14^{10} 7$, $24^{10} 6$, $3 4^{10} 5$, $1 4^{11} 3$, $2^2 4^{11}$, or $4^{12}$.
There are fifteen solutions $g_1$, from
\begin{align*}
( 1, 4)( 2, 7)( 3,10)( 5,13)( 6,16)( 8,18)( 9,19)(11,21)\\
(12,14)(15,22)(17,24)(20,25)(23,28)(26,31)(27,34)(29,37)\\
(30,40)(32,42)(33,43)(35,45)(36,38)(39,46)(41,48)(44,47)
\end{align*}
to
\begin{align*}
( 1, 4)( 2, 7)( 3,10)( 5,13)( 6,16)( 8,19)( 9,22)(11,25)\\
(12,14)(15,28)(17,30)(18,31)(20,34)(21,37)(23,39)(24,40)\\
(26,32)(27,29)(33,43)(35,45)(36,46)(38,48)(41,47)(42,44).
\end{align*}

In every case, $g_1 g_2$ is of type $4^88^2$, and in every case, the group is imprimitive, preserving a decomposition of 
$\{1,2,\ldots,48\}$
into a set of $24$ pairs on which it acts transitively.

 \end{proof}

\begin{prop}
\label{e8}
Let $g_1$, $g_2$, $g_3$, and $g_4$ be elements of $\mathsf{S}_{120}$ of type $2^{60}$, $3^{40}$, $5^{24}$, and 
$1^{118}2$
respectively such that $g_1 g_2 g_3 g_4$ is of type $1^{118}2$. If $\langle g_1,g_2,g_3,g_4\rangle$ is transitive, then
there exists a partition of $\{1,2,\ldots,120\}$ into $2$-element sets such that all $g_i$ respect the partition and $g_1 g_2 g_3$ and $g_4$ send each $2$-element set to itself.

\end{prop}

\begin{proof}
We proceed as before, with $g_2=(1,2,3)\cdots (118,119,120)$.  As $g_1 g_2$ must be the product of an element of $C_{5^{24}}$ and two transpositions, a full list of the possibilities for its type goes as follows:
$5^{20}(10)^2 $, $5^{21}(15)$, $145^{21}(10)$,  $23 5^{21}(10)$,  $15^{22}9$,  $25^{22}8$, $35^{22}7$, $45^{22}6$, $145^{23}$, $235^{23}$, $5^{24}$.
The number of representatives of orbits of $C_{2^{60}}$ is now quite large, at least $\frac{120!}{2^{60}60! 3^{40} 40!} \approx 7\cdot 10^{31}$, so the calculation is feasible only if we can avoid enumerating all the possibilities explicitly.

To accomplish this, we use a recursive procedure which determines, for each initial string $g_1^\circ = (a_1,b_1)\cdots(a_k,b_k)$, all possible choices $g_1$ which give rise to a (not necessarily transitive) solution.
The key step is to rule out  enough initial strings $g_1^\circ$ for which no solution is possible  that the computation is manageable. 
Since $g_2$ is fixed, the initial string  $g_1^\circ$ determines the partially defined function $g_1^\circ g_2$  
from $\{1,2,\ldots,120\}$ to itself. 

We consider the graph with vertex set $\{1,2,\ldots,120\}$ for which edges are given by $g_1 g_2$ whenever it is defined.
Each connected component of this graph  is either a complete cycle (which will be a cycle of $g_1 g_2$, for all $g_1$ with initial string $g_1^\circ$) or a sequence (which will be a consecutive 
part of a cycle of $g_1 g_2$, for all $g_1$ with initial string $g_1^\circ$.)  By comparing the lengths of these components to the target type of $g_1 g_2$, we can quickly eliminate those $g_1^\circ$ for which a solution is already impossible, thus reducing the number of cases to be checked to a few million.
For instance, if 
$$g_1^\circ = (1,4)(2,5)(3,7)(6,10)(8,12)(9,13)(11,16)(14,18),$$
then we obtain the following graph (with isolated vertices omitted):
\vskip 10pt

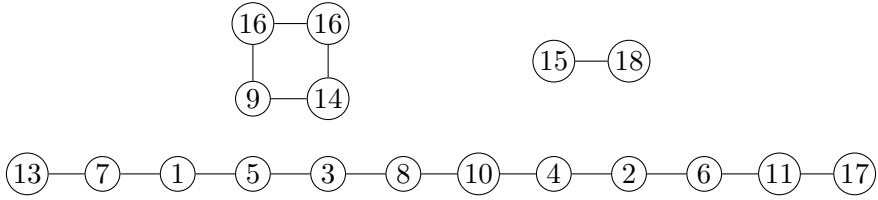
\begin{tikzpicture}[
    vertex/.style={circle, draw, minimum size=1.2em, inner sep=1pt},
    edge/.style={-}
]

% Define vertex positions to minimize edge crossings
\node[vertex] (13) at (0,1){13};
\node[vertex] (7) at (1,1){7};
\node[vertex] (1) at (2,1){1};
\node[vertex] (5) at (3,1){5};
\node[vertex] (3) at (4,1){3};
\node[vertex] (8) at (5,1){8};
\node[vertex] (10) at (6,1){10};
\node[vertex] (4) at (7,1){4};
\node[vertex] (2) at (8,1){2};
\node[vertex] (6) at (9,1){6};
\node[vertex] (11) at (10,1){11};
\node[vertex] (17) at (11,1){17};
\node[vertex] (9) at (3,2){9};
\node[vertex] (14) at (4,2){14};
\node[vertex] (16) at (4,3){16};
\node[vertex] (12) at (3,3){16};
\node[vertex] (15) at (7,2.5){15};
\node[vertex] (18) at (8,2.5){18};

%3,8,10,4,2,6,11,17
%(9 14 16 12)
%18 15
%
\draw[edge] (13) -- (7);
\draw[edge] (7) -- (1);
\draw[edge] (1) -- (5);
\draw[edge] (5) -- (3);
\draw[edge] (3) -- (8);
\draw[edge] (8) -- (10);
\draw[edge] (10) -- (4);
\draw[edge] (4) -- (2);
\draw[edge] (2) -- (6);
\draw[edge] (6) -- (11);
\draw[edge] (11) -- (17);
\draw[edge] (9) -- (14);
\draw[edge] (14) -- (16);
\draw[edge] (16) -- (12);
\draw[edge] (12) -- (9);
\draw[edge] (15) -- (18);

\end{tikzpicture}

Since $g_1 g_2$ cannot have both a $4$-cycle and a cycle of length $\ge 12$, there are no solutions with initial string $g_1^\circ$.

We find that there are sixty-six transitive solutions in all, of which
\begin{gather*}
(  1,  4)(  2,  7)(  3, 10)(  5, 13)(  6, 16)(  8, 19)(  9, 22)( 11, 25)( 12, 28)( 14, 30)( 15, 31)( 17, 34)\\
( 18, 20)( 21, 37)( 23, 40)( 24, 26)( 27, 43)( 29, 46)( 32, 49)( 33, 35)( 36, 52)( 38, 54)\\
( 39, 41)( 42, 55)( 44, 57)( 45, 47)( 48, 50)( 51, 58)( 53, 60)( 56, 61)( 59, 64)( 62, 67)\\
 63, 70)( 65, 73)( 66, 76)( 68, 79)( 69, 82)( 71, 85)( 72, 88)( 74, 90)( 75, 91)( 77, 94)\\
( 78, 80)( 81, 97)( 83,100)( 84, 86)( 87,103)( 89,106)( 92,109)( 93, 95)( 96,112)( 98,114)\\
( 99,101)(102,115)(104,117)(105,107)(108,110)(111,118)(113,120)(116,119)
\end{gather*}
is the first.
In every case, $g_1 g_2$ is of type $5^{20}(10)^2$, and in every case, the group is imprimitive, preserving a decomposition of 
$\{1,2,\ldots,120\}$
into a set of $60$ pairs, on which it acts transitively.
\end{proof}

We can now prove the main theorem.

\begin{proof}
By Theorem~\ref{local}, we must be in one of the cases $D_n$, $E_6$, $E_7$, or $E_8$.
In the $D_n$ cases, by Proposition~\ref{dn}, $L$ and $M$ contain subfields $L_0$ and $M_0$ respectively which are degree $2$ extensions of $K$ ramified only at $x_1$ and $x_2$.
Up to isomorphism as $K$-algebras, there is a unique such extension, so $L_0\cong M_0$ as $K$-algebras, and $L_0\otimes_K M_0$ is the product of two fields and therefore contains zero-divisors.
Since $L\otimes_K M\supset L_0\otimes_K M_0$ is a field, this is impossible.

In the cases $E_6$, $E_7$, and $E_8$, by Propositions \ref{e6}, \ref{e7}, and \ref{e8} respectively, we see that the Galois group $G_L$ of the splitting field $L^s$ of $L/K$
acts transitively but imprimitively on the set $X$ of $24$, $48$, or $120$ $K$-embeddings of $L$ in $L^s$.  Let $H_L = \Gal(L^s/L)$, so $H_L$ is the stabilizer of one point in this action.
We have seen that $G_L$ fixes a partition of $X$ into a set $Y$ of pairs, and $G_L$ acts transitively on $Y$.  Let $L_0$ denote the fixed field of $L^s$ with respect to the subgroup of $G_L$ which acts trivially on $Y$.
Transpositions must act trivially on any non-trivial partition of $X$, so $g_4$ and $g_1g_2g_3$ fix $L_0$, and $g_1,g_2,g_3$ map to elements $\bar g_i$ of $\Gal(L_0/L)$ which multiply to $1$ and generate a transitive subgroup of the permutation group of $Y$.  Let $d_i$ denote the order of $\bar g_i$, which must 
be either the order of $g_i$ or half the order of $g_i$.
Thus, $d_1\in \{1,2\}$, $d_2=3$, and $d_3=3$ in the $E_6$ case, $d_3\in \{2,4\}$ in the $E_7$ case, and $d_3=5$ in the $E_8$ case.
We must have $d_1=2$, since otherwise the group $\langle \bar g_1,\bar g_2,\bar g_3\rangle = \langle \bar g_2\rangle$ has order $3$, which is impossible since it acts transitively on a set of at least $12$ elements.

The group $\langle \bar g_1,\bar g_2,\bar g_3\rangle$ is a quotient of the von Dyck group
$$D(d_1,d_2,d_3) := \langle z_1,z_2,z_3| z_1^{d_1},z_2^{d_2},z_3^{d_3},z_1z_2z_3\rangle,$$
which, from the classification of spherical reflection groups, is of order $12$, $24$, $6$, and $60$ respectively for $(d_1,d_2,d_3)$ equal to $(2,3,3)$, $(2,3,4)$, $(2,3,2)$, and $(2,3,5)$ respectively.
Thus, we have $d_2 = 4$ in the $E_7$ case and $\langle \bar g_1,\bar g_2,\bar g_3\rangle = D(d_1,d_2,d_3)$ in all cases, and this group acts simply transitively on $Y$. It follows that $H_L$ acts trivially on $Y$,
so $L_0\subset L$.  As there is a unique transitive action of a group of order $n$ on an $n$-element set, field $L_0$ is uniquely determined by $x_1,x_2,x_3$.

Repeating the same construction for $M$, we obtain a field $M_0$, also ramified over $x_1$, $x_2$, and $x_3$ with local monodromy $\bar g_1$, $\bar g_2$, and $\bar g_3$.
Therefore, $L_0$ and $M_0$ are isomorphic as $K$-extensions, so $L_0\otimes_K M_0$ is not a field, which contradicts the hypothesis that $L\otimes_K M\supset L_0\otimes_K M_0$ is a field.

\end{proof}

\section{Latt\`es maps}

Let $E$ be an elliptic curve over $\C$ and $n$ a positive integer. The multiplication map $[n]\colon E\to E$ commutes with the $\Z/2\Z$ action on $E$ generated by $z\mapsto -z$, so we obtain the \emph{Latt\`es map},
$\langle n\rangle\colon \P^1\to \P^1$, where $\P^1 = E/\pm 1$.  We denote by $\pi$ the quotient map $E\to \P^1$.  Since the diagram 
$$\begin{CD}
E @>[n]>> E \\
@V\pi VV @VV\pi V \\
\mathbb{P}^1 @>\langle n \rangle>> \mathbb{P}^1
\end{CD}$$
commutes, we have $\deg \langle n\rangle = \deg [n] = n^2$.  If the two sets are not disjoint, then $2P\in \ker [n]$, so $nP\in \ker [2]$, means that $\langle n\rangle(x)$ must be one of the $4$ ramification points of 
$\pi$.

A point $x$ in $\P^1$ corresponds to a pair $\{\pm P\}\subset E$ and $\langle n\rangle(x)$ corresponds to $\{\pm n P\}$.  If $y\in \P^1$ corresponds to $\{\pm Q\}\subset E$, then
$\langle n\rangle (x) = \langle n\rangle (y)$ if and only if $P-Q\in \ker [n]$ or $P+Q\in \ker[n]$.  If $P+\ker[n]$ and $-P+\ker[n]$ are disjoint, then there are $2n^2$ choices of $Q$ and therefore $n^2$ pairs $\pm Q$
which map by $\langle n\rangle$ to $\langle n\rangle (x)$.  Therefore, $\langle n\rangle$ is unramified over $\langle n\rangle (x)$. Conversely, if $\langle n\rangle (x) \in [2]$, then $2P\in \ker[n]$, so
$P+\ker[n]=-P+\ker[n]$.  If $n$ is odd, for each element $P\in \ker[2]$, $[n]^{-1}(P)$ contains one element of $\ker[2]$, and if $n$ is even, $[n]^{-1}(0)$ contains four elements of $\ker[2]$ and $[n]^{-1}(P)$ contains no elements of $\ker[2]$ for $P\in \ker[2]\setminus\{0\}$.
We conclude that for $n>2$ and $P\in \ker[2]$, the number of orbits of $\{\pm 1\}$ on $[n]^{-1}(P)$ is strictly less than $n^2$, so $\langle n\rangle$ is ramified over $\{\pm P\}\in \P^1$.

\begin{prop}
If $\ell$ and $m$ are relatively prime integers $>2$, $K\cong L\cong M$ is the function field of $E$, and the inclusions $K\hookrightarrow L$ and $K\hookrightarrow M$ correspond to multiplication by $\ell$ and $m$ respectively, then $L/K$ and $M/K$ are each ramified over the same four points, and $L\otimes_K M$ is isomorphic to the function field of $E$, regarded as an extension of $K$ via multiplication by $\ell m$.
\end{prop}

\begin{proof}
We have already seen that $L/K$ and $M/K$ are ramified over the same four points as $\pi$. The diagram 
$$\begin{CD}
\P^1 @>\langle \ell\rangle>> \P^1 \\
@V\langle \ell\rangle VV @VV\langle m\rangle V \\
\P^1 @>\langle m \rangle>> \P^1
\end{CD}$$
commutes, so if $L\otimes_K M$ is a field at all, it is unirational and therefore rational by L\"uroth's theorem.

To prove that the fiber product of two copies of $\P^1$ mapping to $\P^1$ by $\langle \ell\rangle$ and $\langle m\rangle$ is irreducible, it suffices to construct a surjective map from $E$ to this fiber product. The diagram
$$\begin{CD}
E @>\pi\circ [ \ell]>> \P^1 \\
@V\pi\circ [ \ell] VV @VV\langle m\rangle V \\
\P^1 @>\langle m \rangle>> \P^1
\end{CD}$$
commutes.  If $x,y\in \P^1$ satisfy $\langle m\rangle(x) = \langle \ell\rangle(y)$, choose $P\in \pi^{-1}(x)$, $Q\in \pi^{-1}(y)$.
As $\pi(mP) = \pi(\ell Q)$, replacing $Q$ by $-Q$ if necessary, we may assume $m P = \ell Q$.  Identifying $E$ with $\C/\Lambda$, we choose $\alpha,\beta\in \C$  such that  $\alpha+\Lambda$ and $\beta+\Lambda$
correspond to $P$ and $Q$ respectively.  As $mP=\ell Q$, we have $m \alpha - \ell\beta \in \Lambda$.
By the Chinese Remainder Theorem, there exist $\gamma,\delta\in \Lambda$ such that $m\alpha - \ell \beta = m\gamma-\ell\delta$. Letting $R\in E$ correspond to $\frac{\alpha-\gamma}\ell = \frac{\beta-\delta}m$,
we see that $\ell R = P$ and $m R = Q$.
\end{proof}

\end{document}